\documentclass[a4paper, reqno]{article} 

\usepackage[utf8]{inputenc}
\usepackage[T1]{fontenc}

\usepackage[english]{babel}

\usepackage[margin=1.2in]{geometry}
\usepackage{enumitem}
\usepackage[bookmarks = true]{hyperref}
\usepackage{amsmath, amsthm, amssymb, bbm}
\usepackage[abbrev]{amsrefs}
\usepackage{unicode-math}
\usepackage[small,nohug,heads=vee]{diagrams}
\diagramstyle[labelstyle=\scriptstyle]
\theoremstyle{plain}
\newtheorem{thm}{Theorem}[section]
\newtheorem{lem}[thm]{Lemma}
\newtheorem{prop}[thm]{Proposition}
\newtheorem{cor}[thm]{Corollary}

\theoremstyle{definition}

\newtheorem{defn}[thm]{Definition}
\newtheorem{ex}[thm]{Example}

\numberwithin{equation}{section}

\newcommand{\overbar}[1]{\mkern 1.5mu\overline{\mkern-1.5mu#1\mkern-1.5mu}\mkern 1.5mu}

\newcommand{\mr}{\mathrm}

\newcommand{\Tor}{\operatorname{Tor}}

\newcommand{\Int}{\operatorname{Int}}

\newcommand{\Aut}{\operatorname{Aut}}

\newcommand{\Hom}{\operatorname{Hom}}
\newcommand{\End}{\operatorname{End}}

\newcommand{\Lie}{\operatorname{Lie}}
\newcommand{\Spec}{\operatorname{Spec}}

\newcommand{\Spf}{\operatorname{Spf}}

\newcommand{\colim}{\operatorname{colim}}

\newcommand{\id}{\mathrm{id}}

\newcommand{\len}{\operatorname{len}}

\renewcommand{\Im}{\operatorname{Im}}

\newcommand{\tensor}{\otimes}
\newcommand{\os}{\overset}

\newcommand{\iso}{\cong}

\newcommand{\mbC}{\mathbb{C}}

\newcommand{\mbE}{\mathbb{E}}
\newcommand{\mbF}{\mathbb{F}}

\newcommand{\mbL}{\mathbb{L}}

\newcommand{\mbN}{\mathbb{N}}

\newcommand{\mbQ}{\mathbb{Q}}
\newcommand{\mbR}{\mathbb{R}}

\newcommand{\mbV}{\mathbb{V}}

\newcommand{\mbX}{\mathbb{X}}

\newcommand{\mbZ}{\mathbb{Z}}

\newcommand{\mcC}{\mathcal{C}}

\newcommand{\mcE}{\mathcal{E}}
\newcommand{\mcF}{\mathcal{F}}

\newcommand{\mcI}{\mathcal{I}}
\newcommand{\mcJ}{\mathcal{J}}

\newcommand{\mcM}{\mathcal{M}}
\newcommand{\mcN}{\mathcal{N}}
\newcommand{\mcO}{\mathcal{O}}

\newcommand{\mcT}{\mathcal{T}}
\newcommand{\mcU}{\mathcal{U}}

\newcommand{\mcX}{\mathcal{X}}

\newcommand{\mcZ}{\mathcal{Z}}

\begin{document}

\title{Local Constancy of Intersection Numbers}
\author{Andreas Mihatsch}
\date{January 20, 2021}
\maketitle

\section{Introduction}
Intersection problems of formal schemes, parametrized by a locally profinite set such as the $\mbQ_p$-points of some variety, are quite common in arithmetic geometry. The aim of this article is to show that, under fairly general assumptions, such intersection numbers vary locally constantly with the parameter. The result applies in the context of the Arithmetic Fundamental Lemma (AFL) of W. Zhang and contributes to its recent proof over $\mbQ_p$, cf. \cite{Z_proof_AFL}. In fact, this is the main motivation for our work and the article can be read as an appendix to \cite{Z_proof_AFL}. Our result might also prove useful in related situations; for example, it similarly applies in the context of the Arithmetic Transfer Conjectures of \cite[§10--§12]{RSZ_regular_formal_moduli}.

Recall that the proof of the AFL in \cite{Z_proof_AFL} relies on global methods and that there already is a local constancy result, \cite[Theorem 5.5]{Z_proof_AFL}, which is needed to approximate $\mbQ_p$-parameters by $\mbQ$-parameters. However, it is only stated for the dense open of so-called \emph{strongly} regular semi-simple elements and, consequently, the AFL is proved with this additional restriction. Our result yields the local constancy for \emph{all} regular semi-simple elements and removes the restriction. We briefly formulate this corollary, but refer to \cite{Z_proof_AFL} for most of its notation; only the intersection number $\Int(g)$ will come up in this article and we provide full details in Section \ref{sectAFL}.
\begin{cor}[{{to \cite[Theorem 15.1]{Z_proof_AFL} and Corollary \ref{cor afl grp loc const}}}]\label{cor afl Qp holds}
The AFL holds for $F_0 = \mbQ_p$ and $p\geq n$. In other words, for $γ\in S_n(F_0)$ regular semi-simple with match $g\in U(\mbV_n)(F_0)$,
$$\partial \mr{Orb}(γ,1_{S_n(O_{F_0})}) = -\mr{Int}(g)\cdot \log q.$$
\end{cor}

Let us now describe the contents of the paper. The first aim, achieved in Section 2, is to define the product $\mcM_S := S\times \mcM$ of a profinite set $S$ and a locally noetherian formal scheme $\mcM$. It is obtained by gluing the affine formal schemes $\Spf C(S,A)$, where $\Spf A\subseteq \mcM$ is open affine and $C(S,A)$ the adic ring of continuous maps from $S$ to $A$. A \emph{family of closed formal subschemes of $\mcM$ parametrized by $S$} is then, by definition, simply a closed formal subscheme of $\mcM_S$.

In Section 3, we consider several such families $\mcZ_1,\ldots,\mcZ_r\subseteq \mcM_S$ and assume that $\mcM$ lies over the formal spectrum of a complete discrete valuation ring, $f\colon \mcM\rightarrow\Spf W$. Consider the following conditions:
\begin{enumerate}[leftmargin=*, label=(Z\arabic*), topsep=2pt, itemsep=2pt]
\item Each $\mcO_{\mcZ_i}$ is a perfect complex. By definition this means that it is, when viewed in the derived category of $\mcO_{\mcM_S}$-modules, locally on $\mcM_S$ represented by a finite complex of finite free modules, see \cite[\href{https://stacks.math.columbia.edu/tag/0BCJ}{Tag 0BCJ}]{Stacks}.
\item Each $\mcO_{\mcZ_i}$ is \emph{flat over $S$} in the following sense. Let $\mcM_s\subset\mcM_S$ denote the fiber above $s\in S$. It is the closed formal subscheme corresponding to the evaluation-in-$s$ maps $C(S,A)→A$ on rings. Then, for all such $s$ and all $j\geq 1$, we demand
$$\mcT\! or_j^{\mcO_{\mcM_S}}(\mcO_{\mcZ_i}, \mcO_{\mcM_s}) = 0.$$
\item Each fiber $\mcZ_s$ of the intersection $\mcZ := \mcZ_1\cap\ldots\cap \mcZ_r$ is a proper scheme over $\Spec W$ and the function $s\mapsto \mcZ_s$ is locally constant. By a fiber criterion, see Corollary \ref{cor fib crit}, this is equivalent to the seemingly stronger statement that there exist an open covering $S = \cup_i\, U_i$ and schemes $Z_i\subseteq \mcM$ which are proper over $\Spec W$ such that $U_i\times_S\mcZ = (Z_i)_{U_i}$.
\end{enumerate}
Note that for all homological algebra definitions and constructions in this article (e.g. (Z1) and (Z2) above), formal schemes are considered as ringed spaces only, the topology of the structure sheaf being irrelevant. Now assume that (Z1) and (Z3) hold. Then the fiber-wise intersection number
\begin{equation}\begin{aligned}\label{defInt}
\Int(s) :=\ & χ(\mcO_{\mcZ_{1,s}}\tensor^\mbL_{\mcO_{\mcM}}\cdots \tensor^\mbL_{\mcO_{\mcM}}\mcO_{\mcZ_{r,s}})\\
:=\ & \sum_{i}(-1)^i\mr{len}_WR^if_*(\mcO_{\mcZ_{1,s}}\tensor^\mbL_{\mcO_{\mcM}}\cdots \tensor^\mbL_{\mcO_{\mcM}}\mcO_{\mcZ_{r,s}}) \in \mbZ
\end{aligned}
\end{equation}
is defined and we are interested in its variation in $s$. On the other hand, we can directly study the complex
$$K:=\mcO_{\mcZ_1}\tensor^\mbL_{\mcO_{\mcM_S}}\cdots \tensor^\mbL_{\mcO_{\mcM_S}}\mcO_{\mcZ_r}.$$
It is perfect by (Z1) and we may consider the Euler-Poincaré characteristics $χ(K_s)$ of its fibers $K_s := K\tensor^\mbL_{\mcO_{\mcM_S}} \mcO_{\mcM_s},\ s\in S$. Proposition \ref{propLocconst} states that $χ(K_s)$ is automatically locally constant. The key observation here, which was explained to us by P. Scholze, is that the schemes $(Z_i)_{U_i}$ from (Z3) are \emph{coherent}. This implies that the cohomology sheaves of $K$ are finitely presented $\mcO_{\mcZ_1\cap\ldots\cap \mcZ_r}$-modules and hence locally constant over $S$.

The quantities $\Int(s)$ and $χ(K_s)$ agree if also (Z2) holds. This additional condition is imposed since the formation of $K_s$ involves a derived base change, but the fibers of the cycles $\mcZ_{i,s}$ are taken in the plain sense. The following is our main result in the general setting.
\begin{thm}\label{thmInt}
Let $f\colon \mcM\rightarrow \Spf W$ and $\mcZ_1,\ldots,\mcZ_r\subseteq \mcM_S$ satisfy (Z1), (Z2) and (Z3) as above. Then the function $\Int(s)$ is locally constant on $S$.
\end{thm}

Section 4 is devoted to its application in the context of the AFL. We define the family version of the AFL intersection problem and check the conditions of the theorem: (Z1) and (Z2) can be verified since the relevant cycles are only of two very specific types, namely divisors or translates of a constant family. (Z3) is essentially known, though not explicit in the literature, so we included it as Lemma \ref{lem loc const afl}. It is noteworthy that abstract arguments reduce this to a statement in Dieudonné theory only, a simplification that can be expected in other contexts as well.

\subsection*{Acknowledgment}
This article originated from the ARGOS seminar in Bonn where W. Zhang's proof of the AFL was studied during summer 2019. P. Scholze gave a sketch of the arguments presented here and I thank him heartily for allowing me to work out his ideas. I further thank J. Anschütz, M. Rapoport and W. Zhang for helpful comments on earlier versions of the paper. I am moreover indebted to the referee for valuable suggestions and for pointing out the insufficiency of the original proof of Proposition \ref{propKR}.

Part of this work was carried out during my stay at MIT with financial support from the Deutsche Forschungsgemeinschaft (DFG, Grant MI 2591/1-1). I am most grateful to MIT for providing such a welcoming working environment and to DFG for making my visit possible.

\section{Formal schemes and profinite sets}
The aim of this section is to define the product $\mcM_S := S\times\mcM$ of a profinite set $S$ and a locally noetherian formal scheme $\mcM$ and to prove some properties of this construction. We begin by examining the relevant rings.

For a profinite set $S$ and a ring $A$, we denote by $LC(S,A)$ the ring of locally constant functions from $S$ to $A$. It can also be described as follows. After choosing a presentation $S = \lim_iS_i$ of $S$ as a filtered inverse limit of finite sets $S_i$ (cf. \cite[\href{https://stacks.math.columbia.edu/tag/08ZY}{Tag 08ZY}]{Stacks}), we get $LC(S,A) = \colim_i LC(S_i,A)$ where the $i$-th term identifies with the product $\prod_{s\in S_i}A$. Note that this colimit is also filtered and that all its transition maps are flat. We similarly define $LC(S,M)$ for any $A$-module $M$ and obtain a similar description $LC(S,M) = \colim_iLC(S_i,M)$.
\begin{lem}\label{lemFlat}
\begin{enumerate}[leftmargin=*, label=(\arabic*), topsep=2pt, itemsep=2pt]
\item The natural map induces an identification $LC(S,A)\tensor_AM = LC(S,M).$
\item The map $A→ LC(S,A)$ is flat, and so are the evaluation maps for $s\in S$,
$$LC(S,A)→A,\ f\mapsto f(s).$$
\end{enumerate}
\end{lem}
\begin{proof}
Statement (1) is clearly true for $S$ finite. Recall that $\tensor$-products commutes with colimits, so the general case of (1) follows by taking the colimit over all isomorphisms $LC(S_i,A)\tensor_AM\iso LC(S_i,M)$. Statment (2) similarly follows from the finite case, but by applying Lemma \ref{lemLimitflat} below. We isolate this argument since we are going to use it again.
\end{proof}

\begin{lem}\label{lemLimitflat}
Let $B = \colim_iB_i$ be a filtered colimit of rings $B_i$ along flat transition maps. Let $ϕ\colon B→C$ be a ring homomorphism such that all restrictions $ϕ\vert_{B_i}$ are flat. Then the natural maps $B_i→B$ as well as $ϕ$ are flat.
\end{lem}
\begin{proof}
For fixed $i$, the set $\{j\mid j\geq i\}$ is also filtered and $B=\colim_jB_j$. Using that $\tensor$-products and colimits commute and that filtered colimits are exact, we see that $B\tensor_{B_i}-$ is an exact functor. In order to prove the flatness of $ϕ$, we recall that it is enough to show $\Tor_1^B(C,B/J)=0$ for all \emph{finitely generated} ideals $J\subseteq B$. Any such ideal is of the form $J_iB$ for some $i$ and some finitely generated ideal $J_i\subseteq B_i$. Furthermore, the just established flatness implies that $J = B\tensor_{B_i}J_i.$
It follows that $C\tensor_B(J→B) = C\tensor_{B_i}(J_i→B_i)$ is injective.
\end{proof}

From now on we assume that $A$ is noetherian and adic with ideal of definition $I$. We denote by $C(S,A)$ the ring of continuous maps from $S$ to $A$. We similarly define $C(S,M)$ whenever $M$ is an $I$-adically complete $A$-module.
It follows directly from the definitions that these coincide with the $I$-adic completions of $LC(S,A)$ and $LC(S,M)$, respectively. Then both $C(S,A)$ and $C(S,M)$ are themselves $I$-adically complete because $I$ is finitely generated, cf. \cite[\href{https://stacks.math.columbia.edu/tag/05GG}{Tag 05GG}]{Stacks}.

\begin{prop}\label{propAdic}
Let $A$ and $S$ be as above.\vspace{-4pt}
\begin{enumerate}[leftmargin=*, label=(\arabic*), topsep=2pt, itemsep=2pt]
\item If $M\os{α}{→}N\os{β}{→}P$ is an exact sequence of finitely generated $A$-modules, then the sequence $C(S,M)→C(S,N)→C(S,P)$ is also exact.
\item For a finitely generated $A$-module $M$, the natural map $C(S,A)\tensor_AM→C(S,M)$ is an isomorphism.
\item The maps $A→C(S,A)$ and $LC(S,A)→C(S,A)$ are flat.
\end{enumerate}
\end{prop}
\begin{proof}
We prove (1). Let $n\colon S→ N$ be a continuous map such that $β(n)=0$. Then $n$ takes values in $\mr{Im}(α)$ and we need to show that it can be lifted to a continuous map $S→ M$. By Artin--Rees, the subspace topology on $\mr{Im}(α)$ is the $I$-adic one, so $n\in C(S,\mr{Im}(α))$. The map $LC(S,M)→ LC(S,\mr{Im}(α))$ is surjective and so, by \cite[\href{https://stacks.math.columbia.edu/tag/0315}{Tag 0315}]{Stacks}, the map on $I$-adic completions $C(S,M)→ C(S,\mr{Im}(α))$ is surjective as well. It follows that a lift exists.

To prove (2), we choose an exact sequence $0→Q→A^r→M→0$ and obtain a commutative diagram with exact rows
\newarrow{eqArrow} =====
\begin{diagram}
&& C(S,A)\tensor_AQ & \rTo & C(S,A)\tensor_AA^r & \rTo & C(S,A)\tensor_AM & \rTo & 0\\
&& \dTo && \deqArrow && \dTo\\
0 & \rTo & C(S,Q) & \rTo & C(S,A^r) & \rTo & C(S,M) & \rTo & 0.
\end{diagram}
The middle vertical equality implies that the right vertical arrow is surjective. Repeating the argument for $Q$ instead of $M$ shows that the left vertical arrow is also surjective. By a diagram chase, the right vertical arrow is then also injective.

We finally come to (3). The flatness of $A→C(S,A)$ follows from (1) and (2) and the fact that it can be checked on finitely generated $A$-modules. The same argument shows flatness of all the maps $LC(S_i,A)→C(S,A)$ and we conclude by applying Lemma \ref{lemLimitflat}.
\end{proof}

Recall that a ring is coherent if each of its finitely generated ideals is finitely presented, cf. \cite[\href{https://stacks.math.columbia.edu/tag/05CV}{Tag 05CV}]{Stacks}. All the rings $LC(S_i,A)$ are coherent (even noetherian) since we assumed $A$ to be so. It follows that $LC(S,A)$ is also coherent, since this property is stable under taking filtered colimits along flat transition maps.

\begin{cor}\label{corCoh}
Let $A$ and $S$ be as above and let $J\subseteq A$ be an open ideal. Then $C(S,A)/J C(S,A)$ is isomorphic to $LC(S,A/J)$. In particular, it is a coherent ring.
\end{cor}
\begin{proof}
By part (3) of Proposition \ref{propAdic}, $J C(S,A) = C(S,A)\tensor_AJ$. By part (2), it further equals $C(S,J)$. Applying part (1) shows $C(S,A)/C(S,J) = C(S,A/J)$ which equals $LC(S,A/J)$.
\end{proof}

The ring $C(S,A)$ is itself an adic ring with ideal of definition $IC(S,A)=C(S,I)$ and we consider its formal spectrum $(\Spf A)_S := \Spf C(S,A)$ as an affine formal scheme in the sense of EGA, i.e. as a locally topologically ringed space, see \cite[D\'efinition 10.1.2]{EGAI}. Note that $|(\Spf A)_S| = S\times|\Spf A|$ on the level of topological spaces. The next lemma shows that this product description is compatible with the formation of rings of sections.
\begin{lem}\label{lemLocalization}
Let $A$ and $S$ be as above and let $f\in A$ be any. Then the natural map
$$\widehat{C(S,A)[f^{-1}]}→C(S,\widehat{A[f^{-1}]})$$
is an isomorphism, where $\widehat{(\,\cdot\,)}$ denotes $I$-adic completion.
\end{lem}
\begin{proof}
Both sides are $I$-adically complete, so it is enough to check this modulo $I^n$ for every $n$. On the left hand side, we get $\widehat{C(S,A)[f^{-1}]}/I^n = C(S,A)/I^n[f^{-1}]$ which equals $LC(S,A/I^n)[f^{-1}]$ by Corollary \ref{corCoh}. On the right hand side, the same corollary yields
$C(S,\widehat{A[f^{-1}]})/I^n = LC(S,A[f^{-1}]/I^n)$ and the claim follows.
\end{proof}

\begin{defn}
For $\mcM$ a locally noetherian formal scheme, we denote by $\mcM_S$ the formal scheme obtained by gluing all $(\Spf A)_S$, for $\Spf A\subseteq \mcM$ open affine, according to Lemma \ref{lemLocalization}.
\end{defn}
We have $|\mcM_S|=S\times|\mcM|$ on underlying topological spaces and the construction is functorial in both $\mcM$ and $S$. Two special cases are the projection map $\mcM_S→\mcM$ and the inclusion of fiber maps $\mcM_s\hookrightarrow \mcM_S$ for $s\in S$, which are closed immersions. The projection map is flat in general by Proposition \ref{propAdic} (3). The fiber inclusions are only flat in special situations, e.g. if $\mcM$ is a scheme, in which case we may apply \ref{lemFlat} (2).

In preparation for the next result, we now make some remarks concerning closed formal subschemes. By definition, a closed immersion $ι\colon \mcZ→\mcM_S$ is a morphism of formal schemes such that $|ι|\colon |\mcZ|→|\mcM_S|$ is a closed immersion and $ι^{-1}\colon \mcO_{\mcM_S}→ι_*\mcO_{\mcZ}$ a surjection of sheaves of topological rings. The latter is supposed to mean that $\mcO_\mcZ$ also carries the quotient topology.\footnote{For example, $\Spf \mbZ_p→\Spec \mbZ_p$ is not a closed immersion.} As $\mcO_\mcZ$ is separated, the kernel of $ι^{-1}$ is then a closed ideal sheaf of $\mcO_{\mcM_S}$. Every morphism of affine formal schemes is induced from a map on global sections, see \cite[Proposition 10.2.2]{EGAI}, so $ι$ is locally isomorphic to maps of the form $\Spf C(S,A)/J → \Spf C(S,A)$ with $J$ a closed ideal of $C(S,A)$. Conversely, this yields a local description of closed immersions.
\begin{ex}\label{ex closed ideals}
\begin{enumerate}[leftmargin=*, label=(\arabic*), topsep=2pt, itemsep=2pt]
\item Let $J\subseteq A$ be an ideal. Then $C(S,J)\subseteq C(S,A)$ is closed. Indeed, if $(j_n)_n$ is a converging sequence of continuous functions $j_n\colon S→J$, then $\lim j_n(s) \in J$ since $J$ is closed in $A$. Proposition \ref{propAdic} implies $C(S,A)/C(S,J) = C(S, A/J)$. This shows that the natural map
$$\Spf C(S,A/J) → \Spf C(S,A)$$
is a closed immersion. In other words, the ``base extension'' $\mcZ_S→\mcM_S$ of a closed immersion $\mcZ → \mcM$ is also a closed immersion as expected.
\item Consider $S = \mbZ_p$ and $f\in C(S,\mbZ_p)$ the identity map $f(s) = s$. Then $f$ is a non-zero divisor but the generated ideal $(f)$ is not closed. For example $s\mapsto (-1)^{v_p(s)}s \in \overbar{(f)}\setminus (f)$.
\item Assume now that $f\in C(S,A)$ is such that all fibers $f(s)$ are non-zero divisors. In particular $f$ is a non-zero divisor. Artin--Rees implies that for each $s$ there is a constant $k(s)$ with
$$I^{m+k(s)}\cap (f(s)) \subseteq I^mf(s).$$
Lemma \ref{lem1} (2) below shows that $k(s)$ may be chosen uniformly for all $s\in S$; denote a choice of such constant by $k$ and consider a converging sequence $(a_nf)_n,\ a_n\in C(S,A)$. Then $(a_n-a_m)f\in C(S,I^{m+k})$ implies $a_n-a_m\in C(S,I^m)$, which can be checked fiber-wise using that all $f(s)$ are non-zero divisors. It follows that $(a_n)_n$ converges and, hence, that the generated ideal $(f)$ is closed.
\end{enumerate}
\end{ex}

Let us fix an ideal sheaf of definition $\mcI\subseteq \mcO_{\mcM_S}$; existence is guaranteed since we assumed $\mcM$ to be locally noetherian. Then an ideal sheaf $\mcJ\subseteq \mcO_{\mcM_S}$ is closed if and only if $\mcJ = \lim_n \mcJ/(\mcJ\cap \mcI^n)$. Moreover, $\mcJ$ defines a closed formal subscheme (i.e. the image of a closed immersion) if and only if each quotient $\mcJ/(\mcJ\cap \mcI^n)$ is a quasi-coherent ideal sheaf of $\mcO_{\mcM_S}/\mcI^n$. In other words, the closed formal subschemes of $\mcM_S$ are precisely the compatible families of closed subschemes of the $V(\mcI^n)$.

\begin{prop}\label{prop fib crit}
Let $\mcZ\subseteq \mcM_S$ be a closed formal subscheme.\vspace{-4pt}
\begin{enumerate}[leftmargin=*, label=(\arabic*), topsep=2pt, itemsep=2pt]
\item $\mcZ$ is determined by all its fibers $\mcZ_s$ in the following sense. If $\mcZ'\subseteq \mcM_S$ is another closed formal subscheme such that $\mcZ_s\subseteq \mcZ'_s$ for all $s$, then $\mcZ \subseteq \mcZ'$.
\item $\mcZ$ is a scheme if and only if each fiber $\mcZ_s$ is a scheme.
\item If $\mcZ$ is a quasi-compact scheme that is locally defined by finitely many equations, then it is locally constant. By this we mean that there exist an open covering $S = \cup_i\, U_i$ and schemes $Z_i\subseteq \mcM$ such that $U_i\times_S\mcZ = Z_{U_i}$ for all $i$.
\end{enumerate}
\end{prop}
\begin{proof}
All three statements reduce to the affine case $\mcM = \Spf A$. Let $\mcZ = \Spf C(S,A)/J$ with $J$ closed.

For statement (1), say $\mcZ' = \Spf C(S,A)/J'$ with $J'$ closed and fix some ideal of definition $I\subset A$. The quotients $C(S,A)/C(S,I^n)$ identify with $LC(S,A/I^n)$. Its elements are locally constant, so the fiber-wise inclusions $\mcZ_s\subseteq \mcZ'_s$ imply $J'/(J'\cap C(S,I^n)) \subseteq J/(J\cap C(S,I^n))$. Taking the inverse limit over $n$ and using that both $J$ and $J'$ are closed, the inclusion $J'\subseteq J$ follows.

For the non-trivial direction of (2), denote by $J_s\subseteq A$ the defining ideal of $\mcZ_s$ and assume that all $J_s$ are open. We prove the slightly stronger semi-continuity statement that $J_s\subseteq J_{s'}$ for all $s'$ in an open neighborhood of $s$.

Given $s$, pick $g_1,\ldots,g_m\in J$ such that their values $g_1(s),\ldots,g_m(s)$ generate $J_s$. By continuity of the $g_i$, there is an open neighborhood $U$ of $s$ such that
$$(g_i-g_i(s))\vert_{U}\in C(U,IJ_s)\ \ \forall i.$$
In particular, all $g_i\vert_U$ lie in $C(U,J_s)$ and the $g_i(s')$ generate $J_s$ for all $s'\in U$ by the usual Nakayama Lemma. It follows that $J_s\subseteq J_{s'}$ for $s'\in U$ and hence $U\times_S \mcZ \subseteq \Spec LC(U, A/J_s)$ by (1).

We prove (3). If $\mcZ$ is a scheme, then it is contained in $\Spec LC(S,A/I^n)$ for $n$ large enough. If it is moreover defined by finitely many equations, then it has to be locally constant since these functions are locally constant modulo $I^n$.
\end{proof}

\begin{cor}[Fiber Criterion]\label{cor fib crit}
Let $\mcZ\subseteq \mcM_S$ be a closed formal subscheme such that all fibers $\mcZ_s$ are schemes and such that the function $s\mapsto \mcZ_s$ is locally constant. Then $\mcZ$ is locally constant. In particular, the equivalence claimed in the definition of (Z3) holds.
\end{cor}
\begin{proof}
By Proposition \ref{prop fib crit} (2), $\mcZ$ is itself a scheme. Shrinking $S$, we may assume that $Z:=\mcZ_s$ is constant. Then $\mcZ$ and $Z_S$ have the same fibers and thus agree by (1) of the same proposition.
\end{proof}

\section{Intersection numbers in families}
The aim of this section is to prove the local constancy of intersection numbers in profinite families. Let $W$ be a complete DVR and $f\colon \mcM→\Spf W$ a locally noetherian formal scheme.
\begin{defn}
Let $K\in D(\mcM)$ be a perfect complex such that all $H^j(K)$ are supported on some closed formal subscheme $Z\subseteq \mcM$ that is a proper scheme over $\Spec W$. Then we define the \emph{Euler-Poincaré characteristic} of $K$ as
$$χ(K):= \sum_{i\in \mbZ} (-1)^i\len_W R^if_*K.$$
\end{defn}
Let $S$ denote a profinite set and let $f_S\colon\mcM_S→(\Spf W)_S$ be the induced map. We denote the fiber over $s\in S$ by $\mcM_s\hookrightarrow\mcM_S$. For a perfect complex $K\in D(\mcM_S)$, we set $K_s := \mcO_{\mcM_s}\tensor^\mbL_{\mcO_{\mcM_S}}K$.
\begin{prop}\label{propLocconst}
Let $K\in D(\mcM_S)$ be a perfect complex such that all $H^j(K)$ are supported on $Z_S$ for some closed formal subscheme $Z\subseteq \mcM$ that is a proper scheme over $\Spec W$. Then the fiber-wise Euler-Poincaré characteristic $s\mapsto χ(K_s)$ is locally constant.
\end{prop}

\emph{Proof.} We begin by showing that all $H^j(K)$ are finitely presented $\mcO_{Z_S}$-modules. This is a local property, so we may assume the following affine setting. Let $\mcM =\Spf A$, $Z=\Spec A/J_0$ with $J_0$ open, $B=C(S,A)$, $J = BJ_0$ and $K$ a perfect complex of $B$-modules with $JH^j(K)=0$ for all $j$.

The first observation is that the cohomology groups of $K\tensor^{\mbL}_B B/J$ are finitely presented over $B/J$, even coherent. Indeed, $B/J$ is a coherent ring by Corollary \ref{corCoh} and $K\tensor^{\mbL}_B B/J$ a perfect complex of $B/J$-modules, so the next lemma applies:
\begin{lem}\label{lemCohomComplex}
Let $R$ be a ring and $L^\bullet$ a complex of coherent $R$-modules. Then all $H^j(L^\bullet)$ are coherent.
\end{lem}
\begin{proof}
We denote the differentials on $L^\bullet$ by $d^\bullet$ and consider the short exact sequences
$$0→\ker(d^j)→L^j→\Im(d^j)→0 \ \ \ \mr{and}$$
$$0→\Im(d^{j-1})→\ker(d^j)→H^j(L^\bullet)→0.$$
Then $\Im(d^j)$ is a coherent $R$-module since it is a finitely generated submodule of the coherent module $L^{j+1}$. It follows that $\ker(d^j)$ is finitely generated since it is the kernel of a surjection of finitely presented modules. As $\ker(d^j)$ is also a submodule of $L^j$, which is coherent, it is itself coherent. The second exact sequence then implies that $H^j(L^\bullet)$ is coherent as well.
\end{proof}

To show that the $H^j(K)$ are themselves coherent $B/J$-modules we employ an inductive argument. It applies to any $K\in D^-(B)$ such that (1) $JH^j(K) = 0$ for all $j$ and such that (2) $K\tensor_B^\mbL B/J \in D^-_{\mr{coh}}(B/J)$ is representable by a bounded-above complex of coherent $B/J$-modules.

Take $k$ such that $H^i(K) = 0$ for $i > k$. Then the ``top'' cohomology
$$H^{k}(K) \overset{\text{use }(1)}{=} H^{k}(K)\tensor_BB/J \iso H^{k}(K\tensor^{\mbL}_BB/J)$$
is a coherent $B/J$-module by (2). The truncation $L := τ^{\leq k-1}K$ has the property that $H^i(L) = 0$ for $i\geq k$ and $H^i(L) \iso H^i(K)$ for $i \leq k-1$. The induction is now achieved by replacing $K$ by $L$ and $k$ by $k-1$, given we can verify (1) and (2) for $L$. From these (1) is clear. For (2) consider the exact triangle
$$L → K → H^{k}(K) →.$$
It gives rise to an exact triangle
$$L\tensor_B^\mbL B/J → K \tensor_B^\mbL B/J → H^k(K) \tensor_B^\mbL B/J →.$$
Recall that $J = BJ_0$. Since $A$ is noetherian, there is a resolution by finite free $A$-modules $P^\bullet → A/J_0$. Since $A\to B$ is flat, $B\tensor_AP^\bullet → B/J$ is a resolution by finite free $B$-modules. Applying Lemma \ref{lemCohomComplex} to $H^k(K)\tensor_A P^\bullet$ then shows that the rightmost term $H^k(K)\tensor_B^\mbL B/J$ lies in $D^-_{\mr{coh}}(B/J)$. The same is true for the middle term implying $L\tensor_B^\mbL B/J\in D^-_{\mr{coh}}(B/J)$, i.e. (2), concluding the induction.

Returning to the original global setting, it is now established that all $H^j(K)$ are finitely presented $\mcO_{Z_S}$-modules. It follows that there exists a projection to a finite set $S→S_0$ such that all $H^j(K)$ are obtained by pullback from $Z_{S_0}$. Replacing $S$ by its fibers over $S_0$, we may assume $H^j(K) = \mcO_{Z_S}\tensor_{\mcO_Z}\mcF^j$ for certain coherent $\mcO_Z$-modules $\mcF^j$. As $Z_S$ is separated, $Rf_{S,*}H^j(K)$ may be computed with Čech cohomology from any affine cover. In particular, we are free to choose an affine open cover of $Z_S$ of the form $\mcU_S = \{U_{i,S}\}$ for $\mcU =\{U_i\}$ an open affine cover of $Z$. Then, for the corresponding Čech complexes,
$$\check\mcC^\bullet(\mcU_S, H^j(K)) = LC(S,W/π^n)\tensor_{W/π^n} \check\mcC^\bullet(\mcU, \mcF^j),$$
where $π\in W$ is a uniformizer and $n$ is such that $π^n\mcO_Z = 0$.
The flatness of $W/π^n→LC(S,W/π^n)$ from Lemma \ref{lemFlat} (2) then implies that $R^if_{S,*}H^j(K) = LC(S,W/π^n)\tensor_{W/π^n}R^if_*\mcF^j$. Moreover, the flatness of the evaluation maps $LC(S,W/π^n)→W/π^n,φ\mapsto φ(s)$ implies the base change isomorphism
$$(Rf_{S,*}H^j(K))_s \iso Rf_*(H^j(K)_s)=Rf_*\mcF^j.$$
We now consider the cohomology-to-hypercohomology spectral sequence that computes $R^{i+j}f_{S,*}K$ from all $R^if_{S,*}H^j(K)$. Its existence implies both that all $R^{i+j}f_{S,*}K$ are finitely presented $LC(S,W/π^n)$-modules, hence locally constant over $S$, and that base change at the level of complexes holds,
\begin{equation}\label{eqBCFiber}
(Rf_{S,*}K)_s \iso Rf_*(K_s).
\end{equation}
The local constancy of the Euler-Poincaré characteristic follows.
\qed

\begin{thm}\label{thmLocconst}
Let $\mcZ_1,\ldots,\mcZ_r\subseteq \mcM_S$ be closed formal subschemes that satisfy (Z1), (Z2) and (Z3) from the introduction. Then the fiber intersection number
$$\Int\colon S → \mbZ,\ s\mapsto χ\big(\mcO_{\mcZ_{1,s}}\tensor^\mbL_{\mcO_\mcM}\cdots\tensor^\mbL_{\mcO_\mcM}\mcO_{\mcZ_{r,s}}\big).$$
is locally constant on $S$.
\end{thm}
\begin{proof}
Shrinking $S$ and by (Z3), we can assume $\mcZ_1\cap\ldots\cap \mcZ_r = Z_S$ for some scheme $Z\subseteq \mcM$ that is proper over $\Spec W$. We set
$$K := \mcO_{\mcZ_1}\tensor^\mbL_{\mcO_{\mcM_S}} \cdots \tensor^\mbL_{\mcO_{\mcM_S}} \mcO_{\mcZ_r},$$
which is a perfect complex by (Z1). The $H^j(K)$ have support on $Z_S$ and it follows that $s\mapsto χ(K_s)$ is locally constant by Proposition \ref{propLocconst}. All that is left to show is that
$$K_s \iso \mcO_{\mcZ_{1,s}}\tensor_{\mcO_{\mcM_s}}^\mbL\cdots\tensor^\mbL_{\mcO_{\mcM_s}}\mcO_{\mcZ_{r,s}}.$$
But it is formal that
$$K_s \iso \Big(\big(\mcO_{\mcZ_1}\tensor^\mbL_{\mcO_{\mcM_S}} \mcO_{\mcM_s}\big)\tensor^\mbL_{\mcO_{\mcM_s}}\cdots \tensor^\mbL_{\mcO_{\mcM_s}}\big(\mcO_{\mcZ_r}\tensor^\mbL_{\mcO_{\mcM_S}} \mcO_{\mcM_s}\big)\Big).$$
This follows e.g. from \cite[\href{https://stacks.math.columbia.edu/tag/08YU}{Tag 08YU}]{Stacks} or by a direct verification after choosing free resolutions of the $\mcO_{\mcZ_i}$ locally. Using the flatness assumption (Z2), we see that actually
$$\big(\mcO_{\mcZ_i}\tensor^\mbL_{\mcO_{\mcM_S}} \mcO_{\mcM_s}\big)\iso \big(\mcO_{\mcZ_i}\tensor_{\mcO_{\mcM_S}} \mcO_{\mcM_s}\big)$$
and the proof is complete.
\end{proof}

\begin{ex}\label{exZ2}
Assume that $\mcZ\subseteq \Spf C(S,A)$ is defined by the closed ideal $J$ and that $J = (f_1,\ldots,f_r)$ is generated by a regular sequence. Then $\mcO_\mcZ$ is quasi-isomorphic to the Koszul complex $K$ of the $f_1,\ldots,f_r$. Its fiber $K_s$ in $s$ is the Koszul complex of $f_1(s),\ldots,f_r(s)$. So (Z2) holds for $\mcZ$ if and only if $K_s$ is quasi-isomorphic to $\mcO_{\mcZ_s}$ for all $s$, i.e. if the $f_1(s),\ldots,f_r(s)$ again form a regular sequence for all $s$. For example, (Z2) holds in the situation of Example \ref{ex closed ideals} (3).
\end{ex}

\section{Application to the AFL}
\label{sectAFL}
Recall that there are two variants of the AFL: For the \emph{group version}, a ``diagonal'' cycle is intersected with a translate; for the \emph{semi-Lie algebra} version, a diagonal cycle is intersected with a translate \emph{and} a so-called KR-divisor. It is the local constancy in the slightly more complicated semi-Lie algebra setting that is needed in \cite{Z_proof_AFL} and this is the main result of this last section. The local constancy for the group version will be a special case, to be found at the very end.

For $p$ an odd prime, let $F/F_0$ be an unramified quadratic extension of $p$-adic local fields with Galois conjugation $σ$ and let $\breve F$ be the completion of a maximal unramified extension of $F$. We denote the corresponding rings of integers by $\mcO_{F_0}\subset \mcO_F\subset \mcO_{\breve F}$ and denote the residue field of $\breve F$ by $\mbF$. The following definitions are taken directly from \cite[§3.1]{Z_proof_AFL}.

(1) A \emph{hermitian $\mcO_F$-module over $T$}, where $T$ is an $\Spf \mcO_{\breve F}$-scheme, is a triple $(X,ι,λ)$ of the following kind. $X$ is a supersingular strict formal $\mcO_{F_0}$-module over $T$, $ι\colon \mcO_F→\End(X)$ an $\mcO_F$-action and $λ\colon X\iso X^\vee$ a principal polarization such that $λ\circ ι(a) = ι(σ a)^\vee\circ λ$. The $\mcO_F$-action yields a decomposition $\Lie X = L_0\oplus L_1$, where $L_0$ (resp. $L_1$) is the eigenspace on which $\mcO_F$ acts naturally (resp. via composition with $σ$). The locally constant pair $(\mathrm{rk}_{\mcO_T} L_0,\mathrm{rk}_{\mcO_T}L_1)$ is called the \emph{signature} of $(X,ι,λ)$.

(2) For each $n$, we fix a hermitian $\mcO_F$-module $\mbX_n = (\mbX_n,ι_{\mbX_n},λ_{\mbX_n})$ of signature $(1,n-1)$, the so-called \emph{framing object}. We also fix hermitian $\mcO_F$-modules $\mbE$ and $\overbar \mbE$ of signatures $(1,0)$ and $(0,1)$, respectively, over $\mbF$. Then there is an isogeny $\mbX_n→\mbE\times\overbar{\mbE}^{n-1}$ and
$$V_n := \Hom^0_F(\overbar\mbE,\mbX_n)$$
is an $n$-dimensional hermitian $F$-vector space, the space of so-called \emph{special homomorphisms}. The group of self-quasi-isogenies $G_n:=\Aut(\mbX_n,ι_{\mbX_n},λ_{\mbX_n})$ acts unitarily on $V_n$ and identifies $G_n$ with $U(V_n)(F_0)$.

(3) We let $\mcN_n$ denote the formal scheme over $\Spf \mcO_{\breve F}$ whose $T$-valued points are quadruples $(X,ι,λ,ρ)$ up to isomorphism, where $(X,ι,λ)$ is a hermitian $\mcO_F$-module over $T$ and $ρ$ a quasi-isogeny of height $0$ to the framing object over the special fiber $\overbar T := \mbF\tensor_{\mcO_{\breve F}} T$,
$$ρ\colon (X,ι,λ)\times_T\overbar{T}→(\mbX_n,ι_{\mbX_n},λ_{\mbX_n})\times_{\Spec \mbF} \overbar{T}.$$
It is locally formally of finite type and formally smooth of relative dimension $n-1$ over $\Spf \mcO_{\breve F}$. We henceforth fix some $n\geq 2$ and drop the index ``$n$''.

(4) Let $\overbar\mcE$ be the unique deformation of $\overbar \mbE$ over $\Spf \mcO_{\breve F}$. For $u\in V$, we denote by $\mcZ(u)\subseteq \mcN$ the closed formal subscheme of points $(X,ι,λ,ρ)$ with the property that the quasi-homomorphism $ρ^{-1}u\colon \overbar{\mbE}\times_{\Spec \mbF} \overbar T →X\times_T\overbar T$ lifts to a homomorphism $\overbar \mcE\times_{\Spf \mcO_{\breve F}}T→X$. By \cite[Proposition 3.5]{KR_local}, this is a relative Cartier divisor over $\Spf \mcO_{\breve F}$ if $u\neq 0$, the so-called \emph{KR-divisor} of $u$.

(5) A pair $(g,u)\in G\times V$ is \emph{regular semi-simple} if $\{g^iu\}_{i\geq 0}$ generates $V$ as $F$-vector space. For regular semi-simple $(g,u)$, we consider the following intersection number from \cite[Equation (3.9)]{Z_proof_AFL},
\begin{equation}\label{eqAFL}
\Int(g,u):=χ\big(\mcO_Δ\tensor_{\mcO_{\mcM}}^\mbL \mcO_{(\id\times g)Δ}\tensor_{\mcO_\mcM}^\mbL \mcO_{\mcZ(u)\times_{\Spf \mcO_{\breve F}} \mcN}\big).
\end{equation}
Here $Δ\subset \mcM:=\mcN\times_{\Spf \mcO_{\breve F}}\mcN$ is the diagonal.
Finiteness of this number is implied by two properties. First, $\mcO_Δ$ and $\mcO_{\mcZ(u)}$ are perfect complexes. This follows immediately from the regularity of $\mcM$. Second, the schematic intersection $Δ\cap (\id\times g)Δ\cap (\mcZ(u)\times_{\Spf \mcO_{\breve F}} \mcN)$ is a proper scheme over $\Spec \mcO_{\breve F}$, see e.g. \cite[Lemma 6.1]{M_rel_unit_RZ}.

\begin{thm}\label{thmAFL}
The intersection number $\Int(g,u)$ is a locally constant function on the regular semi-simple elements $(G\times V)_\mr{rs}$.
\end{thm}
\emph{Proof.}
We would like to apply Theorem \ref{thmLocconst} from the previous section. So let $S$ be a profinite set and $(g,u)\colon S→(G\times V)_{\mr{rs}}$ a continuous map. We first define the family versions of the three cycles in the definition of $\Int(g,u)$. These will be closed formal subschemes of $\mcM_S$. 

(6) The first cycle is the diagonal in $\mcM_S = \mcN_S\times_{(\Spf \mcO_{\breve F})_S} \mcN_S.$ It agrees with the base change $Δ_S$ as in example \ref{ex closed ideals} (1). Since $\mcM_S→\mcM$ is flat (Proposition \ref{propAdic}), $Δ_S$ satisfies the two conditions (Z1) and (Z2).

(7) The second cycle is defined as the translate $(\mr{id}\times g)Δ_S$. To this end, we now explain how $g$ induces an automorphism of $\mcN_S$: This formal scheme represents the functor which takes an $\mcO_{\breve F}$-algebra $R$ in which $p$ is nilpotent to the set of isomorphism classes of tuples $(X,ι,λ,ρ,t)$ of a point $(X,ι,λ,ρ)\in \mcN(\Spec R)$ and an $\mcO_{\breve F}$-algebra map $t\colon C(S,\mcO_{\breve F})→R$. (Continuity of $t$ is automatic since it factors over some $LC(S,\mcO_{\breve F}/p^n)$.)
The map $g$ defines a quasi-isogeny of $\mbX\times_{\Spec \mbF}LC(S,\mbF)$ which can be described as follows. After multiplying by a power of $p$, we may assume that $g(s)$ is an isogeny for all $s\in S$. On each truncation $\mbX[p^m]$, $g$ then defines a locally constant family of endomorphisms. These are compatible and define $g$ over $LC(S,\mbF)$. Denoting by $\overbar t :=1_{\mbF}\tensor_{\mcO_{\breve F}}t$ the special fiber of $t$, the automorphism of $\mcN_S$ is now given by composition in $ρ$ as usual, $g(X,ι,λ,ρ,t):=(X,ι,λ,(\Spec \overbar t)^*(g)\circ ρ,t)$. 
The cycle $(\mr{id}\times g)Δ_S$ again satisfies (Z1) and (Z2), since both these properties are stable under automorphisms over $(\Spf \mcO_{\breve F})_S$.

(8) We define the third cycle as the product $\mcZ(u)\times_{(\Spf \mcO_{\breve F})_S}\mcN_S$, where $\mcZ(u)\subset \mcN_S$ is the family version of the usual KR-divisors. To define the latter, observe that $u$ defines a quasi-homomorphism
$$\overbar\mbE\times_{\Spec \mbF}\Spec LC(S,\mbF)→\mbX\times_{\Spec \mbF}\Spec LC(S,\mbF)$$
just as $g$ did in (7). By \cite[Proposition 2.9]{RZ_book}, there is a closed formal subscheme $\mcZ(u)\subseteq \mcN_S$ parametrizing those tuples $(X,ι,λ,ρ,t)$ such that $ρ^{-1}(\Spec \overbar t)^*(u)$ lifts to a homomorphism of $p$-divisible groups $\overbar{\mcE}\times_{\Spf \mcO_{\breve F}} R → X$.

\begin{prop}\label{propKR}
Let $u\colon S→V\setminus \{0\}$ be a continuous map. Then $\mcZ(u)\subset \mcN_S$ is a Cartier divisor which is flat over $S$. In particular, $\mcZ(u)\times_{(\Spf \mcO_{\breve F})_S}\mcN_S$ satisfies (Z1) and (Z2).
\end{prop}
\emph{Proof.}
We begin with an auxiliary result on Cartier divisors on noetherian formal schemes. Let $A$ be an $I$-adically complete noetherian ring and set $\mcU = \Spf A$ as well as $U_m := \Spec A/I^m$.
Consider a principal Cartier divisor $\mcZ = V(f) \subseteq \mcU$. By Artin--Rees there is $k\geq 1$ such that for all $m\geq 0$
\begin{equation}\label{eq AR1}
I^{m+k}\cap (f) = I^m(I^k\cap (f)).
\end{equation}
\begin{lem}\label{lem1}
The following further properties hold.
\begin{enumerate}[leftmargin=*, label=(\arabic*), topsep=2pt, itemsep=2pt]
\item For all $m\geq 1$, any defining equation for $\mcZ\cap U_{m+k}$ lifts to a defining equation for $\mcZ$. In particular, any two defining equations for $\mcZ\cap U_{m+k}$ differ by a unit.
\item Let $V(g)\subseteq \mcU$ be another principal Cartier divisor such that $V(g)\cap U_{k+1} = \mcZ \cap U_{k+1}$. Then Artin--Rees holds with the same $k$ for $g$,
\begin{equation}
\label{eq AR2}
I^{m+k}\cap (g) = I^m(I^k\cap (g)),\ \ m\geq 0.
\end{equation}
\item Any Cartier divisor $\mcZ'\subseteq \mcU$ with $\mcZ'\cap U_{k+1} = \mcZ \cap U_{k+1}$ is principal. 
\end{enumerate}
\end{lem}
\begin{proof}
To prove (1), let $h\in A$ be such that $(h) + I^{m+k} = (f) + I^{m+k}$ with $m\geq 1$. Then there are $s,t\in A$ such that $f = sh\mod I^{m+k}$ and $h = tf\mod I^{m+k}$. Hence $(1-st)f \in I^{m+k}\cap (f) = I^m(I^k\cap (f))$. Since $f$ is a non-zero divisor, $(1-st)\in I^m$. It follows that $st\in 1+ I^m$ is invertible which concludes (1).

For (2), we show \eqref{eq AR2} by inducting over $c = 0,\ldots,m$ in the statement
$$I^{m+k}\cap (g) \subseteq I^c(I^k\cap (g)),$$
the case $c=0$ being tautological. So assume $xg\in I^{m+k}$ can be written as $\sum t_ix_ig$ with $t_i\in I^c$ and $x_ig\in I^k$. Since $g$ is a non-zero divisor this just means $x = \sum t_ix_i$. Our assumptions imply that $g = f + h$ for some $h\in I^{k+1}$, leading to the relation
$$\sum t_ix_i(f + h) \in I^{m+k}.$$
Then $\sum t_ix_ih \in I^{c+k+1}$, so for $c<m$
$$xf = \sum t_ix_if \in I^{c+k+1}\cap (f) = I^{c+1}(I^k \cap (f)).$$
Since $f$ is a non-zero divisor, this means one can write
$$x = \sum s_jy_j$$
with $s_j\in I^{c+1}$ and $y_jf\in I^k$. Then also $y_jg \in I^k$ for all $j$, so
$$xg = \sum s_jy_jg \in I^{c+1}(I^k\cap (g)).$$

We turn to (3). By assumption $\mcZ' \cap U_{k+1} = \mcZ \cap U_{k+1}$ is defined by a single equation, namely $g_{k+1} := f$. We argue by induction on $m\geq 1$ that $\mcZ'\cap U_{m+k} = V(g_{m+k})$ for compatible $g_{m+k}$. Let $\mcU = \cup_{i}\, \mcU^{(i)}$ be an affine open covering such that all $\mcZ' \cap \mcU^{(i)} = V(g^{(i)})$ are principal. The Artin--Rees relation for $f$ \eqref{eq AR1} is stable under localization and completion ($A$ is noetherian), so by (2) the $g^{(i)}$ satisfy an Artin--Rees type relation for the same $k$,
$$(IA_i)^{m+k} \cap g^{(i)}A_i = (IA_i)^m((IA_i)^k\cap g^{(i)}A_i),\ \ m\geq 0.$$
By (1) applied to the $V(g^{(i)})$, an element $g_{m+k}$ can be lifted to a defining equation $g^{(i)}_{m+k+1}$ of $\mcZ'\cap \mcU^{(i)}\cap U_{m+k+1}$. Again by (1) applied to the intersections $\mcU^{(i)}\cap \mcU^{(j)}$ these differ by units on overlaps, defining a $1$-cocycle with values in $\mcO_{U_{m+k+1}}^\times$. Its cohomology class lies in the image of
$$H^1(U_{m+k+1}, I^{m+k}/I^{m+k+1})→H^1(U_{m+k+1}, \mcO_{U_{m+k+1}}^\times).$$
But the left hand side cohomology group vanishes since $U_{m+k+1}$ is affine. The lifts $g^{(i)}_{m+k+1}$ may hence be chosen compatibly, defining $g_{m+k+1}$, concluding induction and proof.
\end{proof}

Assume now that the above $\mcU$ is an affine open of $\mcN$. The aim will be to apply the previous lemma to two KR-divisors $\mcZ(v)\cap \mcU$ and $\mcZ(v')\cap \mcU$. The next lemma shows that for $v$ and $v'$ close enough, the condition from (2) and (3) is met.

\begin{lem}\label{lem2}
Given $m$, there is an $\mcO_F$-lattice $Λ\subseteq V$ such that $U_m \subseteq \mcZ(v)$ for all $v\in Λ$.
\end{lem}
\begin{proof}
Consider a single $v\in V$ and denote the universal tuple over $\mcN$ by $(\mcX,ι,λ,ρ)$. The map
$$ρ^{-1}\circ v\colon \mcE \times_{\Spf \mcO_{\breve F}} U_m → \mcX\vert_{U_m}$$
being a quasi-homomorphism means that there is some integer $r$ (depending on $v$ and $U_m$) such that
$$ρ^{-1}\circ p^r v\in \Hom(\mcE \times_{\Spf \mcO_{\breve F}} U_m, \mcX).$$
Applying this to an $F$-basis of $V$ shows that there is a lattice $Λ\subseteq V$ (depending on $U_m$) such that $U_m \subseteq \mcZ(v)$ for all $v\in Λ$.
\end{proof}

We finally return to the original question for a family $u\colon S→V\setminus\{0\}$. The claim ($\mcZ(u)$ being a Cartier divisor in $\mcN_S$) is local, so we may restrict attention to $\mcZ(u)\cap \mcU_S$. Fix $s_0\in S$. Shrinking $\mcU$ we may assume $\mcZ(u(s_0))\cap \mcU = V(f_0)$ for some $f_0$. Let $k$ be an index as in Lemma \ref{lem1} (1) for $f_0$, let $Λ\subseteq V$ be a lattice for $U_{k+1}$ as in Lemma \ref{lem2}. Then for all $v \in u(s_0) + Λ$ and all $m\geq 1$, $\mcZ(v)\cap \mcU$ is principal by Lemma \ref{lem1} (3) and any defining equation for $\mcZ(v)\cap U_{m+k}$ may be lifted to a defining equation of $\mcZ(v)$ by (2).

Shrinking $S$ to an open compact neighborhood of $s_0$, we may assume $u(s)\in u(s_0)+Λ$ for all $s\in S$. The intersections $\mcZ(u)\cap (U_{m+k})_S,\ m\geq 1$, are locally constant by Lemma \ref{lem2} and, in particular, defined by a single equation. Any choice of equation may be lifted to a defining equation for $\mcZ(u)\cap (U_{m+k+1})_S$ by Lemma \ref{lem1} (1). By induction and limit-taking, one finds $f\in C(S,A)$ such that $V(f) \cap (U_{m+k})_S = \mcZ(u)\cap (U_{m+k})_S$ for all $m$. The fibers $f(s)$ define the divisors $\mcZ(u(s))\cap \mcU$ and are therefor non-zero divisors. Then Example \ref{ex closed ideals} (3) applies and shows that $C(S,A)\cdot f$ is a closed ideal. This proves that $\mcZ(u)$ is a Cartier divisor.

A product argument or the observation that all arguments similarly apply to $\mcZ(u)\times_{\Spf \mcO_{\breve F}} \mcN_S$ shows that the product is also a Cartier divisor. In particular it satisfies (Z1). Property (Z2) is the statement that $\mcZ(u)$ is a Cartier divisor whose fibers $\mcZ(u(s))$ are also Cartier divisors (cf. Example \ref{exZ2}). The proof is now complete.\qed

Set $\mcZ(g,u) := Δ_S\cap (\id\times g)Δ_S\cap (\mcZ(u)\times_{(\Spf \mcO_{\breve F})_S} \mcN_S)$ for the following.

\begin{lem}\label{lem loc const afl}
The fiber-wise intersection $\mcZ(g,u)_s$ is locally constant. In particular, the three cycles from (6)--(8) satisfy (Z3).
\end{lem}
\begin{proof}
We already know that each $\mcZ(g,u)_s$ is a proper scheme (see (5) above), so $\mcZ(g,u)$ is itself a scheme by Proposition \ref{prop fib crit} (2). Moreover, it is locally defined by finitely many equations since this is true for each of the three individual cycles. By Proposition \ref{prop fib crit} (3), all that is left to check is its quasi-compactness. This is a statement about the maximal contained reduced subscheme only. It reduces further to checking the local constancy of the $\mbF$-points $\mcZ(g,u)_s(\mbF)$ since $\mcN_{\mr{red}}$ is locally of finite type over $\Spec \mbF$. We use covariant relative Dieudonné theory (i.e. for strict formal $\mcO_{F_0}$-modules) to prove this. More details on Dieudonné theory in this context may be found in \citelist{\cite{M_rel_unit_RZ}*{§2} \cite{KR_local}*{§2} \cite{RTZ_AFL}}.

Let $(N,τ)$ be the $F_0$-isocrystal of $(\mbX,ι_\mbX,λ_\mbX)$, which is $2n$-dimensional $\breve F$-vector spaces together with an $F_0$-linear Frobenius $τ$, a polarization $λ\colon N\times N→\breve F$ and an \emph{additional} $F$-action $ι$. The latter makes it into a free rank $n$ module over $F\tensor_{F_0}\breve F$. By \emph{Dieudonné lattice}, we mean a self-dual $\mcO_F\tensor_{\mcO_{F_0}}\mcO_{\breve F}$-lattice $M\subset N$ that is stable under the Frobenius $τ$, the Verschiebung $πτ^{-1}$ ($π$ some fixed uniformizer of $\mcO_{F_0}$) and has $\mcO_F$ acting with signature $(1,n-1)$ on its Lie algebra $M/(πτ^{-1}M)$. These lattices are in bijection with $\mcN(\mbF)$. The choice of an $\mcO_F\tensor_{\mcO_{F_0}} \mcO_{\breve F}$-generator $m$ of the Dieudonné lattice of $\overbar{\mbE}$ defines an embedding $V\hookrightarrow N,\ v\mapsto v(m)$ and we identify $V$ with its image. The embedding is $G$-equivariant, isometric up to a scalar in $\mcO_{\breve F}^\times$ and such that $M\in \mcZ(v)(\mbF)$ if and only if $v\in M$.

Identifying the diagonal $Δ$ with $\mcN$, we see that $\mcZ(g,u)_s(\mbF)$ is in bijection with $g(s)$-stable Dieudonné lattices that contain $u(s)$. In particular, $\mcZ(g,u)_s = \emptyset$ if $g(s)$ does not have integral characteristic polynomial and the locus of such $s$ is open and closed. We may hence assume $g(s)$ to have integral characteristic polynomial for all $s$. Then $L(s):=\mcO_F[g(s)]u(s)$ is a locally constant family of lattices (by definition of regular semi-simpleness) and we may assume it to be constant, say constantly $L$. Any Dieudonné lattice $M\in \mcZ(g,u)_s(\mbF)$ then satisfies $\breve L\subseteq M\subseteq \breve L^\vee$, where $\breve L = (\mcO_F\tensor_{\mcO_{F_0}}\mcO_{\breve F})\cdot L$ and where $\breve L^\vee$ is the dual $\mcO_{\breve F}$-lattice. The condition of being $g(s)$-stable then only depends on the class of $g(s)$ in $\End(V)$ modulo $\Hom(L^\vee, L)$ and this class is locally constant. The lemma follows.
\end{proof}

Summing up, we defined three families of cycles that satisfy both (Z1), (Z2) and whose fibers over $S$ are precisely the cycles from the definition of $\Int(g,u)$. By Lemma \ref{lem loc const afl}, they also satisfy (Z3) and an application of Theorem \ref{thmLocconst} finishes the proof of Theorem \ref{thmAFL}.\qed

(9) For the group version of the AFL, we fix some $u_0$ such that $(u_0,u_0) = 1$, where $(\ ,\ )$ is the hermitian form on $V$. Let us call $g\in G$ \emph{regular semi-simple} if $(g,u_0)$ is regular semi-simple in the previous sense. For such $g$, we set $\Int(g) := \Int(g,u_0)$, which is the relevant intersection number, see \cite[§3]{Z_proof_AFL}. We immediately obtain the
\begin{cor}\label{cor afl grp loc const}
The group version intersection number $\Int(g)$ is locally constant on $G_{\mr{rs}}$.
\end{cor}


\end{document}